\documentclass[11pt]{amsart}

\usepackage{latexsym,amssymb,amsmath,graphics}
\usepackage[dvips]{graphicx}

\def\kk{{\mathcal K}}

\def\ltwo{{L^2(\mathbb{R})}}

\def\ffi{\varphi}

\def\dst{\displaystyle}

\def\N{{\mathbb{N}}}

\def\R{{\mathbb{R}}}

\def\Z{{\mathbb{Z}}}

\newcommand{\norm}[1]{{\left\|{#1}\right\|}}

\newcommand{\abs}[1]{{\left|{#1}\right|}}
\newcommand{\scal}[1]{{\left\langle{#1}\right\rangle}}

\newtheorem{lemma}{Lemma}[section]
\newtheorem{proposition}[lemma]{Proposition}
\newtheorem{theorem}[lemma]{Theorem}
\newtheorem{corollary}[lemma]{Corollary}

\newtheorem{examplenum}[lemma]{Example}

\begin{document}

\title{Time-frequency concentration of generating systems}
\author{Philippe Jaming}
\address{Universit\'e d'Orl\'eans\\
Facult\'e des Sciences\\ 
MAPMO\\ BP 6759\\ F 45067 Orl\'eans Cedex 2\\
France}
\email{Philippe.Jaming@univ-orleans.fr}

\author{Alexander M. Powell}
\address{Vanderbilt University\\
Department of Mathematics\\
Nashville, TN 37240\\
USA}
\email{alexander.m.powell@vanderbilt.edu}

\begin{abstract}
Uncertainty principles for generating systems $\{e_n\}_{n=1}^{\infty} \subset \ltwo$ are proven and quantify the interplay
between $\ell^r(\N)$ coefficient stability properties and
time-frequency localization with respect to $|t|^p$
power weight dispersions.
As a sample result, it is proven that there does not exist a unit-norm Schauder basis nor a frame
$\{e_n\}_{n=1}^{\infty}$ for $\ltwo$ such that the two dispersion sequences
$\Delta(e_n)$, $\Delta(\widehat{e_n})$ and one mean sequence $\mu(e_n)$ are bounded.
On the other hand, it is constructively proven that there exists a
unit-norm exact system $\{f_n\}_{n=1}^{\infty}$ in $\ltwo$ for which all four of the
sequences
$\Delta(f_n)$, $\Delta(\widehat{f_n})$, $\mu(f_n)$, $\mu(\widehat{f_n})$
are bounded.
\end{abstract}

\subjclass{}

\keywords{Compactness, exact system, frame, Schauder basis, time-frequency concentration,
uncertainty principle.}

\date{\today}

\thanks{}
\maketitle


\section{Introduction} \label{intro:sec}

The uncertainty principle in harmonic analysis states that a function $f$ and its
Fourier transform $\widehat{f}(\xi) = \int f(t) e^{-2\pi i t \xi},\mbox{d}t$ cannot be simultaneously
too well localized.  Heisenberg's uncertainty principle offers a classical interpretation of this
in terms of means and dispersions.  Heisenberg's inequality 
states that for every 
$f \in \ltwo$ with unit-norm $\|f\|_{2}=1$
$$
\Delta(f) \Delta(\widehat{f}) \geq \frac{1}{4\pi}.
$$
For unit-norm $f \in \ltwo$ the {\em mean} $\mu( \cdot )$ is defined by
$\mu(f) = \int t |f(t)|^2 \,\mbox{d}t$
and the {\em dispersion} $\Delta( \cdot )$ is defined by 
$\Delta(f) = (\int |t - \mu(f)|^2 |f(t)|^2 \,\mbox{d}t)^{1/2}$.  
It is common to say that $f$ is mostly concentrated in the time-frequency plane in
a Heisenberg box centered at $\bigl(\mu(f),\mu(\widehat{f}\,)\bigr)$ with side lengths determined by
$\Delta(f)$ and $\Delta(\widehat{f}\,)$. Heisenberg's inequality states that this box has area at least $1/4\pi$.
For a survey on uncertainty principles, see \cite{FS}.

Heisenberg's uncertainty principle applies to individual functions $f\in\ltwo$.  However, there are also
versions of the  uncertainty principle that constrain the collective time-frequency localization
properties of systems of functions $\{e_n\}_{n=1}^{\infty} \subset \ltwo$ such as  orthonormal bases.
For example, the Balian-Low uncertainty principle, ({\it see e.g.} \cite{BHW}),
states that if the Gabor system 
$$
\mathcal{G}(f,1,1)=\{f_{m,n}\}_{m,n \in \Z} = \{e^{2\pi i m t} f(t-n)\}_{m,n\in\Z}
$$
is an orthonormal basis for $\ltwo$ then the strong uncertainty constraint
$\Delta(f) \Delta(\widehat{f})=\infty$ must hold. The Balian-Low theorem depends crucially on the rigid
structure of Gabor systems.

This article addresses uncertainty principles for general 
generating systems $\{e_n\}_{n=1}^{\infty}$ for $\ltwo$.  Our main results are motivated by a question
of H.S. Shapiro, \cite{Sh}, that asks to characterize the sequences $\Delta(e_n), \Delta(\widehat{e_n}),
\mu(e_n), \mu(\widehat{e_n})$ that arise for orthonormal bases $\{e_n\}_{n=1}^{\infty}$ for $\ltwo$.
The following uncertainty principle for this setting was proven in \cite{Po}:
\begin{eqnarray} 
&&\hbox{\sl If $\{e_n\}_{n=1}^{\infty}$ is an orthonormal basis for $\ltwo$ then
the three}  \label{mvonb} \\ 
&&\hbox{\sl sequences $\mu(e_n), \Delta(e_n), \Delta(\widehat{e_n})$ cannot all be bounded.} \notag
\end{eqnarray}
Further results of this type may be found in \cite{JP,Ma}.
A main goal of this article is to pursue the investigation of this phenomenon
for more general generating systems in $\ltwo$ (such as Schauder bases, frames, or exact systems)
and for generalized dispersions involving a $|t|^p$ power weight (instead of the standard
$|t|^2$ weight).

\subsection*{Overview and main results}
The paper is organized as follows.  Section \ref{background:sec} gives background on
generating systems and generalized power weight dispersions.
Section \ref{compact:sec} discusses the role of compactness in time-frequency concentration results
and contains useful lemmas.

Our first main result, Theorem \ref{mv-mainthm}, is 
proven in Section \ref{mvmain:sec}.  Roughly stated, Theorem \ref{mv-mainthm} shows
 that if a unit-norm system $\{e_n\}_{n=1}^{\infty} \subset \ltwo$ satisfies
$$
\forall f \in \ltwo, \ \ \ A\left( \sum_{n=1}^{\infty} | \langle f, e_n \rangle |^s \right)^{1/s}
\leq \| f \|_{2} \leq B \left( \sum_{n=1}^{\infty} | \langle f, e_n \rangle |^r \right)^{1/r},
$$
then for $qr>2$ (with $p,q>1$, $r,s>0$)
the generalized means and dispersions
$\mu_p(e_n)$, $\Delta_p(e_n)$, $\Delta_q(\widehat{e_n})$
cannot all be bounded.
This improves the dispersion requirements in \cite{Po} and extends
beyond orthonormal bases. 
For example, if $\{e_n\}_{n=1}^{\infty}$ is a Schauder basis for $\ltwo$
 then the standard means and dispersions
 $\mu(e_n), \Delta(e_n), \Delta(\widehat{e_n})$ cannot all be bounded.

Our second main result, Theorem \ref{exact-thm}, appears in 
Section \ref{exact:sec} and shows that the phenomenon in Theorem \ref{mv-mainthm} is not true for
exact systems.  In particular, Theorem \ref{exact-thm} constructively proves that for any
$p,q>1$ there exists
a unit-norm exact system $\{e_n\}_{n=1}^{\infty}$ for $\ltwo$
such that $\Delta_p(e_n)$, $\mu_p(e_n)$, $\Delta_q(\widehat{e_n})$ and $\mu_q(\widehat{e_n})$ are all bounded sequences.

\section{Background} \label{background:sec}

\subsection{Generating systems}
Throughout this section let $H$ be a separable Hilbert space with norm $\| \cdot \|$.
The collection $\{e_n\}_{n=1}^{\infty} \subset H$ is a {\em frame} for $H$ if
there exist constants $0<A,B<\infty$ such that
\begin{equation} \label{frame-ineq}
\forall f \in H, \ \ \ A\|f\|^2 \leq \sum_{n=1}^{\infty} | \langle f, e_n \rangle |^2 
\leq B\|f\|^2.
\end{equation}
A standard result in frame theory states that there exists an associated {\em dual frame}
$\{\widetilde{e_n}\}_{n=1}^{\infty} \subset H$ such that
the following expansions hold with unconditional convergence in $H$
$$
\forall f \in H, \ \ \ f = \sum_{n=1}^{\infty} \langle f, e_n \rangle \widetilde{e_n}
=\sum_{n=1}^{\infty} \langle f,  \widetilde{e_n} \rangle{e_n}.
$$
In general, frames need neither be orthogonal nor provide unique representations.
For example, a union of $k$ orthonormal bases for $H$ is a frame for $H$.  
For further background on frame theory see \cite{Ca,Ch}.

The system $\{e_n\}_{n=1}^{\infty}\subset H$ is {\em minimal} if for every $N \in \N,$
$e_N \notin \overline{span}\{ e_n : n \neq N \}$.  
If $\{e_n\}_{n=1}^{\infty}$ is a minimal system and is a frame for $H$, then we say
that $\{e_n\}_{n=1}^{\infty}$ is a {\em Riesz basis} for $H$.
If $\{e_n\}_{n=1}^{\infty}$ is minimal and is complete in $H$ then we say that
$\{e_n\}_{n=1}^{\infty}$ is an {\em exact} system in $H$.

The system $\{e_n\}_{n=1}^{\infty}$ is a {\em Schauder basis} for $H$ 
if for every $f \in H$ there exists a unique sequence $\{c_n\}_{n=1}^{\infty} \subset \mathbb{C}$
such that
$$
\lim_{N\to \infty} \norm{f - \sum_{n=1}^N c_n e_n} =0.$$
An important result of Gurari\u{\i} and Gurari\u{\i}, see \cite{GG, Ru}, states that if $\{e_n\}_{n=1}^{\infty}$ with $\|e_n\|=1$ is a Schauder
basis for $H$ then there exist $1<r \leq 2 \leq s < \infty$ and constants $0<A, B<\infty$
such that
\begin{equation} \label{gg-cond}
\forall f \in H, \ \ \ 
A \left( \sum_{n=1}^{\infty} | \langle f, e_n \rangle |^s \right)^{1/s}
\leq \|f\| \leq 
B \left( \sum_{n=1}^{\infty} | \langle f, e_n \rangle |^r \right)^{1/r}.
\end{equation}

We shall primarily be interested in the cases where $H=\ltwo$ or where
$H = \overline{span} \{e_n\}_{n=1}^{\infty}$ is the closed linear span of a
collection $\{e_n\}_{n=1}^{\infty} \subset \ltwo$.
Since we shall investigate the applicability of (\ref{mvonb}) to systems
such as frames, Schauder bases, and exact systems, it is useful to note the
following inclusions among different types of generating systems in $\ltwo$:
$$\{\rm orthonormal \ bases \} \subsetneq \{\rm Riesz \ bases\} \subsetneq \{\rm Schauder \ bases \} \subsetneq  \{ {\rm exact \ systems} \},$$
$$ \{\rm Riesz \ bases\} \subsetneq \{\rm frames\}; \ \
\{\rm frames \} \nsubseteq \{\rm Schauder\ bases\}; \ \ 
\{\rm Schauder \ bases \} \nsubseteq \{\rm frames \}.$$
The condition (\ref{gg-cond}) will serve as a key feature in our analysis.  In particular,
(\ref{gg-cond}) holds for all Schauder bases, frames, Riesz bases and orthonormal bases,
but does not in general hold for exact systems.  See \cite{You} for additional information
on generating systems.

\subsection{Generalized means and dispersions}
The standard dispersion defined by
\begin{equation} \label{standard-disp-mean}
\Delta(h) = \left( \int |t - \mu(h)|^2 |h(t)|^2\,\mbox{d}t \right)^{1/2} \ \ \ \hbox{ and } \ \ \ \mu(h) = \int t |h(t)|^2\,\mbox{d}t,
\end{equation}
implicitly makes use of the $|t|^2$ power weight.  In this note we shall consider a generalized
class of dispersions and associated means defined in terms of the $|t|^p$ power weight
with $p>1$.

Let $p>1$. Recall that a function $\ffi$ is strictly convex if 
for every $a \neq b$ and $0< s < 1$ there holds
$\varphi(sa + (1-s)b) < s\varphi(a) + (1-s)\varphi(b).$
A property of convexity states that if $\varphi$ is strictly convex and attains a local minimum
then the local minimum is global and is {\em unique}.

Given $p>1$ the function $a \mapsto |a|^p$ is a standard example of a strictly convex function.
Consequently, for any fixed $t\in\R$, $a\mapsto|t-a|^p$ is also strictly convex.
Now fix $h\in L^2(\R)$ with $\|h\|_{2}=1$, and define
$$
\varphi(a)=\int |t - a|^p |h(t)|^2\,\mbox{d}t.
$$
Since $|a+b|^p=2^p\abs{\frac{1}{2}a+\frac{1}{2}b}^p\leq 2^{p-1}(|a|^p+|b|^p)$, we have that
if $\ffi(a)<\infty$
for some $a\in\R$, then $\varphi(a)<\infty$ for every $a\in\R$. 
Moreover, it may be verified that if $\varphi(a)<\infty$ for some $a\in\R$ then
$\varphi$ is continuous and strictly convex.

Assume that $\|h\|_2=1$ and $\int |t|^p |h(t)|^2 dt < \infty$, so that
$\varphi(a)<\infty$ for every $a\in\R$. 
Then
\begin{eqnarray*}
\varphi(a) + \int |t|^p |h(t)|^2 \,\mbox{d}t
= \int |t|^p |h(t)|^2 \,\mbox{d}t + \int |a-t|^p|h(t)|^2 \,\mbox{d}t \\
\geq 2^{1-p} \int |a|^p|h(t)|^2  \,\mbox{d}t = 2^{1-p}|a|^p
\end{eqnarray*}
which implies that $\lim_{|a| \to \infty} \varphi(a) = \infty$.  Thus, since $\varphi$ is continuous
and strictly convex it follows that $\ffi$ has a unique global minimum.
Given $p>1$ and $\|h\|_2=1$ we may thus define
$$
\Delta_p(h) = \min_{a \in \R} \left(\int |t - a|^p |h(t)|^2 \,\mbox{d}t \right)^{1/2},
$$
and
$$
\mu_p(h) = {\rm arg \thinspace min}_{a\in \R} \left(\int |t - a|^p |h(t)|^2 \,\mbox{d}t \right)^{1/2}.
$$
Note that $\mu_p(h)$ is uniquely defined.
We refer to $\Delta_p(h)$ as the {\em $p$-dispersion} of $h$ and $\mu_p(h)$ as the {\em $p$-mean}
of $h$. 
When $p=2$ it is straightforward to verify that the closed form expressions
$\mu_2(f) = \mu(f)$ and $\Delta_2(f) = \Delta(f)$ from (\ref{standard-disp-mean}) hold.

If $p\leq 1$ then strict convexity no longer holds and the
$p$-mean need not be uniquely defined.  For example, if
$h=2^{-1/2}(\chi_{[-2,-1]}+\chi_{[1,2]})$ then every
$a \in [-1,1]$ is a minimizer for $\dst\varphi(a) = \int |t-a| |h(t)|^2 \,\mbox{d}t$.
We will therefore restrict our attention to $p>1$.

We now state some useful results on $p$-means and $p$-dispersions.
\begin{lemma} \label{lemma:basicdisp}
Let $p>1$.  Suppose that $h \in \ltwo$ satisfies $\|h\|_2=1$ and
that $\Delta_p(h)$ is finite (and hence $\mu_p(h)$ is also finite).
If $h_s$ is defined by $h_s(t)=h(t-s)$, then $\Delta_p(h_s)=\Delta_p(h)$ and 
$\mu_p(h_s)=\mu_p(h)+s$. Moreover, if $h$ is even or odd, then $\mu_p(h)=0$. 
\end{lemma}

\begin{proof}
It is clear that
$\Delta_p(h_s)=\Delta_p(h)$ and $\mu_p(h_s)=\mu_p(h)+s$.
Next assume that $h$ is either even or odd, so that $|h|^2$ is even.
The change of variable $t \mapsto -t$ gives
$$
\int |t - \mu_p(h)|^p |h(t)|^2 \,\mbox{d}t = \int |t + \mu_p(h)|^p |h(t)|^2 \,\mbox{d}t.
$$
Since $\mu_p(h)$ is the unique minimizer of $\varphi(a) = \int |t-a|^p |h(t)|^2 \,\mbox{d}t$
we have $\mu_p(h) = -\mu_p(h)$.  Thus, $\mu_p(h) =0$.
\end{proof}

In general, convergence in $\ltwo$ does not imply convergence of $p$-dispersions.
For example, if we define $\psi_s  = (1-s)^{1/2}\chi_{[-1/2,1/2]} + s^{1/2} \chi_{[s^{-4},1+s^{-4}]}$
and $\psi = \chi_{[-1/2,1/2]},$
then $\lim_{s\to 0} \psi_s = \psi$ holds in $\ltwo$, but
$\lim_{s\to 0} \mu_2(\psi_s) = \infty \neq 0=\mu_2(\psi)$
and
$\lim_{s\to 0} \Delta_2(\psi_s) = \infty \neq \Delta_2(\psi)$.
However, the following holds.

\begin{proposition} \label{prop:dispconv}
Let $p>1$. 
Suppose that $f,g\in \ltwo$ satisfy $\|f\|_{2}=\|g\|_{2}=1$ and that $\Delta_p(f)$ and $\Delta_p(g)$
are finite. 
For each $|\alpha|<1$ define
$\dst h_{\alpha} = \frac{f+ \alpha g}{ \| f + \alpha g \|_2}.$
Then
$$\lim_{\alpha \to 0} \Delta_p(h_\alpha) = \Delta_p(f) \ \ \ \hbox{ and } \ \ \
\lim_{\alpha \to 0 } \mu_p(h_{\alpha}) = \mu_p (f).$$
\end{proposition}

\begin{proof}
{\em Step I.}  We first show that 
\begin{equation} \label{part1eq}
\limsup_{\alpha \to 0} \Delta^2_p(h_{\alpha}) \leq \Delta^2_p(f).
\end{equation}
Noting that $1-|\alpha| \leq \|f + \alpha g\|_2 \leq 1 + |\alpha|,$ we have
\begin{align*}
\Delta^2_p(h_\alpha) & \leq \int |t-\mu_p(f)|^p |h_{\alpha}(t)|^2 \,\mbox{d}t
 \leq (1-|\alpha|)^{-2} \int |t-\mu_p(f)|^p |f(t) + \alpha g(t)|^2 \,\mbox{d}t\\
 & \leq 
 (1-|\alpha|)^{-2} \left[  \Delta_p(f) + |\alpha| \left( \int |t-\mu_p(f)|^p |g(t)|^2 \,\mbox{d}t \right)^{1/2} \right]^2.
\end{align*}
Since $\Delta_p(g)<\infty$ implies $\int |t-\mu_p(f)|^p |g(t)|^2 dt<\infty$, it follows that
(\ref{part1eq}) holds.

\medskip

\noindent {\em Step II.} 
We next show that if $|\alpha|$ is sufficiently small then
\begin{equation} \label{part2eq}
 | \mu_p(h_{\alpha}) - \mu_p(f) | \leq 2 [4 \Delta^2_p(f) ]^{1/p}.
\end{equation}
To begin, by (\ref{part1eq}), take $|\alpha|$ sufficiently small
so that $\Delta^2_p(h_{\alpha}) \leq 2 \Delta^2_p(f)$.
Let $\eta = [4\Delta^2_p(f)]^{1/p}$.  There holds
\begin{align}
\int_{|t-\mu_p(f)|> \eta} |f(t)|^2\,\mbox{d}t\leq
\eta^{-p} \int_{|t-\mu_p(f)|> \eta} |t-\mu_p(f)|^p |f(t)|^2\,\mbox{d}t \leq
 \frac{\Delta^2_p(f)}{\eta^p} = \frac{1}{4}. \label{part2-eq1}
\end{align}
Similarly,
\begin{align}
\frac{1}{2} \leq 1 - \frac{\Delta^2_p(h_{\alpha})}{\eta^p} 
\leq \int |h_{\alpha}(t)|^2\,\mbox{d}t - \int_{|t - \mu_p(h_{\alpha})|> \eta} |h_{\alpha}(t)|^2\,\mbox{d}t
= \int_{|t-\mu_p(h_\alpha)|\leq \eta} |h_{\alpha}(t)|^2\,\mbox{d}t. \label{part2-eq2}
\end{align}
Also note that
\begin{equation} \label{part2-eq3}
\int_{|t-\mu_p(h_{\alpha})|\leq \eta} |h_{\alpha}(t)|^2\,\mbox{d}t 
\leq (1-|\alpha|)^{-2} \left( \int_{|t-\mu_p(h_{\alpha})|\leq \eta}
|f(t)|^2 \,\mbox{d}t + \alpha^2 + 2 |\alpha| \right).
\end{equation}

Now, towards a contradiction, suppose that $|\mu_p(h_{\alpha}) - \mu_p(f)| > 2\eta$ holds.  
This would imply that
\begin{equation} \label{part2-eq4}
\int_{|t-\mu_p(h_{\alpha})|\leq \eta} |f(t)|^2 \,\mbox{d}t \leq \int_{|t-\mu_p(f)|> \eta} |f(t)|^2 \,\mbox{d}t.
\end{equation}
Combining (\ref{part2-eq1}), (\ref{part2-eq2}), (\ref{part2-eq3}), (\ref{part2-eq4}) gives
$$\frac{1}{2} \leq (1-|\alpha|)^{-2} \left( \frac{1}{4} + \alpha^2 + 2|\alpha| \right),$$
which yields a contradiction when $|\alpha|<\dst\frac{\sqrt{38}}{2}-3\sim 0.08$.
Therefore,
$|\mu_p(h_{\alpha}) - \mu_p(f)| \leq 2\eta$ must hold.
In particular, (\ref{part2eq}) holds when $|\alpha|$ is sufficiently small.
\medskip

\noindent {\em Step III.} We next show that
\begin{equation} \label{part3eq}
\Delta^2_p(f) \leq \liminf_{\alpha \to 0} \Delta^2_p(h_{\alpha}).
\end{equation}
If $|\alpha|$ is sufficiently small then (\ref{part2eq}) and Minkowski's inequality imply that 
\begin{align*}
\Delta^2_p(f) &\leq \int |t-\mu_p(h_{\alpha})|^2 |f(t)|^2 \,\mbox{d}t
= \int \bigl|t-\mu_p(h_{\alpha})\bigr|^p\bigl | \|f+\alpha g\|_2 h_{\alpha}(t) - \alpha g(t) \bigr|^2 \,\mbox{d}t\\
 &\leq 
 \left[  (1+|\alpha|) \Delta_p(h_{\alpha}) + |\alpha| \left( \int |t-\mu_p(h_{\alpha})|^p |g(t)|^2\,\mbox{d}t \right)^{1/2} \right]^2\\
&\leq
  \left[  (1+|\alpha|) \Delta_p(h_{\alpha}) +  |\alpha| 2^{\frac{p-1}{2}}\left( \int |t-\mu_p(f)|^p |g(t)|^2\,\mbox{d}t 
+ |\mu_p(f)-\mu_p(h_{\alpha})|^p \right)^{1/2} \right]^2\\
&\leq
  \left[  (1+|\alpha|) \Delta_p(h_{\alpha}) +  |\alpha| 2^{\frac{p-1}{2}}\left( \int |t-\mu_p(f)|^p |g(t)|^2 \,\mbox{d}t 
+2^{p+2}\Delta^2_p(f) \right)^{1/2} \right]^2.
\end{align*} 
Recalling (\ref{part1eq}), this implies that  (\ref{part3eq}) holds.  Moreover, (\ref{part1eq}) and (\ref{part3eq}) yield the theorem's conclusion on convergence of $p$-dispersions:
\begin{equation} \label{dispconv}
\lim_{\alpha \to 0} \Delta_p(h_{\alpha}) = \Delta_p(f).
\end{equation}
\medskip

\noindent {\em Step IV.} Finally, we show that
\begin{equation} \label{part4eq}
\lim_{\alpha \to 0} \mu_p(h_\alpha) = \mu_p(f).
\end{equation}
If $|\alpha|$ is sufficiently small then (\ref{part2eq}) implies that
\begin{align}
 \int |t - \mu_p(h_{\alpha})|^p |g(t)|^2 \,\mbox{d}t & \leq 2^{p-1} \left( | \mu_p(f) - \mu_p(h_{\alpha})|^p
 + \int | t - \mu_p(f)|^p |g(t)|^2 \,\mbox{d}t \right) \notag \\
 & \leq 2^{p-1} \left( 2^p[4\Delta^2_p(f)] + \int | t - \mu_p(f)|^p |g(t)|^2 \,\mbox{d}t \right) < \infty. \label{part4pre1}
\end{align}
Similar computations as for (\ref{part4pre1}) show that if $|\alpha|$ is sufficiently small then
\begin{align}
\left| \int |t \right. & \left. - \mu_p(h_{\alpha})|^p \bigl(\overline{f(t)} g(t)+f(t) \overline{g(t)}\bigr)\,\mbox{d}t \right|
 \leq 2 \int |t - \mu_p(h_{\alpha})|^p |f(t)| |g(t)| \,\mbox{d}t \notag \\ 
&\leq 2 \left( \int | t - \mu_p(h_{\alpha})|^p |f(t)|^2 \,\mbox{d}t \right)^{1/2}
\left( \int | t - \mu_p(h_{\alpha})|^p |g(t)|^2 \,\mbox{d}t \right)^{1/2} \notag \\
& \leq 2^{p} \left( 2^p[4\Delta^2_p(f)] + \int | t - \mu_p(f)|^p |f(t)|^2 \,\mbox{d}t \right)^{1/2} \notag \\
&\qquad  \times \left( 2^p[4\Delta^2_p(f)] + \int | t - \mu_p(f)|^p |g(t)|^2 \,\mbox{d}t \right)^{1/2}
<\infty. \label{part4pre2}
\end{align}
Note that 
\begin{eqnarray} 
\|f+\alpha g\|_2^{2} \Delta^2_p(h_{\alpha}) &=&  \int |t - \mu_p(h_{\alpha})|^p |f(t)|^2 \,\mbox{d}t
 + \alpha^2 \int |t - \mu_p(h_{\alpha})|^p |g(t)|^2 \,\mbox{d}t \label{part4-prelim-eq}\\
  &&+ \alpha \int |t - \mu_p(h_{\alpha})|^p (\overline{f(t)} g(t)+f(t) \overline{g(t)})\,\mbox{d}t, \notag
\end{eqnarray}
We now use that $\Delta_p^2(h_{\alpha}) \to \Delta_p^2(f)$ by (\ref{dispconv}), 
along with (\ref{part4pre1}) and (\ref{part4pre2}), to take limits (as $\alpha\to0$) on both sides of 
(\ref{part4-prelim-eq}) and conclude that the following limit exists and
\begin{equation} \label{alpha-lim-eq}
\lim_{\alpha \to 0} \int |t - \mu_p(h_{\alpha})|^p |f(t)|^2 \,\mbox{d}t  
= \int |t - \mu_p(f)|^p |f(t)|^2 \,\mbox{d}t = \Delta_p^2(f).
\end{equation}

Now, by (\ref{part2eq}), if $|\alpha|$ is sufficiently small then $\mu_p(h_\alpha)$ stays in a fixed compact set.
Thus, there exist sequences $\{\alpha_{k} \}_{k=1}^{\infty}, \{\widetilde{\alpha}_k\}_{k=1}^{\infty}$ 
such that $\alpha_k, \widetilde{\alpha}_k\to0$ and
$$\mu_p(h_{\alpha_k})\to \underline{\mu}:=\liminf_{\alpha\to0}\mu_p(h_\alpha) \in \R \ \ \
\hbox{ and } \ \ \
\mu_p(h_{\widetilde{\alpha}_k})\to \overline{\mu}:=\limsup_{\alpha\to0}\mu_p(h_\alpha) \in \R.$$ 
Note that by (\ref{part2eq})
$$|t - \mu_p(h_{\alpha_k})|^p |f(t)|^2\leq 2^{p-1}\bigl(|t-\mu_p(f)|^p+2^{p+2}\Delta^2(f)\bigr)|f(t)|^2\in 
L^1(\R).$$
Thus, by the dominated convergence theorem
$$
\int |t - \mu_p(h_{\alpha_k})|^p |f(t)|^2 \,\mbox{d}t\to \int |t -  \underline{\mu}|^p|f(t)|^2 \,\mbox{d}t.
$$ 
Therefore by (\ref{alpha-lim-eq})
$$
\int |t -  \underline{\mu}|^p|f(t)|^2 \,\mbox{d}t=\Delta_p^2(f),
$$
and by uniqueness of the global minimum defining $\Delta_p(f)$, we have $\underline{\mu}=\mu_p(f)$.
A similar argument shows that $\overline{\mu}=\mu_p(f)$.
It follows that (\ref{part4eq}) must hold.
\end{proof}

\section{Compactness and time-frequency concentration} \label{compact:sec}

Compactness plays an important role in several aspects of time-frequency analysis.
For example, connections with the uncertainty principle arise in \cite{Sh,Pe}.
We shall require the following compactness result of 
Kolmogorov, Riesz, and Tamarkin, \cite{Ko,Ri,Ta}, also see \cite{DS,Y, HOH, DFG}.  Recall that a set is relatively compact if its closure is compact.

\begin{theorem}\label{th:RT}
Let $\kk$ be a bounded subset of $L^2(\R)$. Then $\kk$ is relatively compact if and only if the following two conditions hold
\begin{itemize}
\renewcommand{\theenumi}{\roman{enumi}}
\item $\kk$ is equicontinuous:
$\dst\int_{\R}|f(t+a)-f(t)|^2\,\mathrm{d}t\to 0$ uniformly in $f\in\kk$ as $a\to 0$,
\item $\kk$ has uniform decay:
$\dst\int_{|t|\geq R}|f(t)|^2\,\mathrm{d}t\to 0$ uniformly in $f\in\kk$ as $R\to+\infty$.
\end{itemize}
\end{theorem}

These properties can  be reformulated in terms of the Fourier transform, see \cite{Pe}, cf. \cite{Sh}.

\begin{theorem} \label{cor:P}
Let $\kk$ be a bounded subset of $L^2(\R)$ and let $\widehat{\kk}=\{\widehat{f}\,:\ f\in\kk\}$. Then $\kk$ is equicontinuous if and only if $\widehat{\kk}$ has uniform decay. In particular, the following conditions are equivalent:
\begin{itemize}
\item $\kk$ is relatively compact,
\item $\kk$ and $\widehat{\kk}$ are both  equicontinuous,
\item $\kk$ and $\widehat{\kk}$ have both uniform decay.
\end{itemize}
\end{theorem}
This gives the following corollary.  The result for $K_{\varphi,\psi}$ appears in \cite{Pe}, cf. \cite{Sh}.
The result for $K_{A}^{p,q}$ appears with $p=q=2$ in \cite{Sh}.  Since
\cite{Sh} is unpublished and since we shall require the result for general values of $p,q$, we include the proof here.

\begin{corollary} \label{cor:sh}
Fix $A>0$, $p,q>1$ 
and $\varphi,\psi \in L^2(\R)$, and define the closed sets
\begin{align*}
\kk_{A}^{p,q} &=\{f\in L^2(\R)\,:\ |\mu_p(f)|\leq A,\ |\Delta_p(f)|\leq A,\ |\mu_q(\widehat{f})|\leq A\ \mathrm{and}\ |\Delta_q(\widehat{f})|\leq A\},\\
\kk_{\varphi,\psi} &=\{f\in L^2(\R)\,:\ |f(t)|\leq|\varphi(t)|\ \mathrm{and}\ |\widehat{f}(t)|\leq|\psi(t)|
\ \mathrm{ for } \ a.e. \ t \in \R\}.
\end{align*}
Then $\kk_A^{p,q}$ and $\kk_{\varphi,\psi}$ are compact subsets of $\ltwo$.
\end{corollary}

\begin{proof}
For $\kk_A^{p,q}$ note that if $|t|>R>2A$ and $|a|<A$, then $|t-a|\geq R/2$.
It follows that if $|\mu_p(f)| \leq A$ and $\Delta_p(f) \leq A$ then as $R\to \infty$
$$
\int_{|t|\geq R}|f(t)|^2\,\mbox{d}t\leq \frac{2^p}{R^p}\int_{|t|\geq R}|t-\mu_p(f)|^p|f(t)|^2\,\mbox{d}t\leq 
\frac{2^pA^2}{R^p} \to 0.
$$
This, along with a similar computation for $\widehat{f}$ and Theorem \ref{cor:P}
shows that  $\kk_{A}^{p,q}$ is compact.

For $\kk_{\varphi,\psi}$ note $\varphi\in L^2(\R)$ implies that 
$\int_{|t|\geq R}|f(t)|^2\,\mbox{d}t\leq\int_{|t|\geq R}|\varphi(t)|^2\,\mbox{d}t\to 0,$
as $R \to \infty$.
This, along with a similar computation for $\widehat{f}$ and Theorem \ref{cor:P}
shows that  $\kk_{\varphi,\psi}$ is compact.
\end{proof} 

We may apply Corollary \ref{cor:sh} to generating systems with appropriate separation properties.
We use the notation $\#I$ to denote the cardinality of a set $I$.

\begin{lemma}
Suppose that $\{e_n\}_{n=1}^{\infty} \subset \ltwo$ with $\|e_n\|_{2}=1$ satisfies one of the
following conditions:
\begin{enumerate}
\item  $\{e_n\}_{n=1}^{\infty}$ converges weakly to zero in 
$H=\overline{span} \{ e_n \}_{n=1}^{\infty}$,
\item  for every $m\in \N$, $\# \{ n \in \N : | \langle e_n, e_m \rangle | \geq 1/2 \} < \infty.$
\end{enumerate}
Then $\{e_n\}_{n=1}^{\infty}$ cannot be contained in a relatively compact
subset of $\ltwo$.
\end{lemma}

\begin{proof} Since condition (1) implies condition (2), it suffices to assume that (2) holds.
Assume towards a contradiction that $\{e_n\}_{n=1}^{\infty}$ is contained in a relatively compact
subset of $\ltwo$.  Then there exists $f\in H$ and a convergent subsequence 
$\{e_{n_k}\}_{n=1}^{\infty}$ such that $e_{n_k} \to f$ strongly in $H$.  Since $\|e_{n_k}\|_{2}=1$
we have that $\|f\|_{2}=1$.

By strong convergence, we have $e_{n_k} \to f$ weakly in $H$.  So
$\langle e_{n_k}, f \rangle \to \langle f, f \rangle =1$. 
Fix $\epsilon>0$.  For $k$ sufficiently large we have $\langle e_{n_k}, f \rangle > 1-\epsilon$
and $\|f - e_{n_k}\|_{2} < \epsilon.$  By (2), for all $j$ sufficiently large (depending on $k$)
one has $| \langle e_{n_k}, e_{n_j} \rangle | < 1/2.$ Thus
$$
1- \epsilon  \leq | \langle e_{n_k}, f \rangle | \leq
| \langle e_{n_k}, e_{n_j} \rangle | + | \langle e_{n_k}, f - e_{n_j} \rangle |
\leq | \langle e_{n_k}, e_{n_j} \rangle | + \|f - e_{n_j} \|_{2} <
1/2 + \epsilon.
$$
This gives a contradiction when $0<\epsilon<1/4$.
\end{proof}

In view of the condition (\ref{gg-cond}) and the frame inequality (\ref{frame-ineq}) we have
the following result.

\begin{corollary}
Suppose $\{e_n\}_{n=1}^{\infty}\subset \ltwo$ with $\|e_n\|_2=1$ is
a frame or Schauder basis for its closed linear span then
$\{e_n\}_{n=1}^{\infty}$ cannot be contained in a relatively compact subset of $\ltwo$.
\end{corollary}

Together with Corollary \ref{cor:sh} this gives the following.

\begin{corollary}\label{thm:tf}
Suppose that $\{e_n\}_{n=1}^{\infty} \subset L^2(\R)$ with $\|e_n\|_{2}=1$ satisfies one of the following two conditions:
\begin{itemize}
\item there exists $A>0$, $p,q>1$ such that, for every $n\in \N$,
$$
|\mu_p(e_n)| \leq A,\ \Delta_p(e_n) \leq A,\ |\mu_q(\widehat{e_n})| \leq A,
\ \Delta_q(\widehat{e_n})\leq A,
$$
\item there exists $\varphi,\psi\in L^2(\R)$ such that
$$
|e_n(t)|\leq |\varphi(t)| \hbox{ and } |\widehat{e_n}(t)|\leq |\psi(t)| \hbox{ a.e.} \ t\in\R.
$$
\end{itemize}
Then $\{e_n\}_{n=1}^{\infty}$ cannot be a Schauder basis or frame for its closed linear span.
\end{corollary}

For our purposes we will need a more quantitative version of Corollary \ref{thm:tf}.

\begin{lemma} \label{pigeon-lem}
Suppose that $\{e_n\}_{n=1}^{k}\subset \ltwo$ with $\|e_n\|_{2}=1$ is such that for every
$m\in \{1,2,\cdots, k\}$,
$$
\# \{ n \in \{1,2,\cdots,k\} : | \langle e_n, e_m \rangle | \geq 1/2 \} < D.
$$
There exists $J \subset \{1,2,\cdots, k\}$ satisfying:
\begin{itemize}
\item $\#J \geq (k -1)/(\lceil D \rceil+1)$,
\item If $m,n \in J$ and $m\neq n$ then $| \langle e_m, e_n \rangle | < 1/2$.
\end{itemize}
\end{lemma}

\begin{proof} We may without loss of generality assume that $D$ is an integer by replacing $D$
with $\lceil D \rceil$.
Reordering $\{e_n\}_{n=1}^k$ if necessary, we may assume that
$| \langle e_n , e_1 \rangle | < 1/2$ for $n=2,\cdots, k-D$.
Reordering $\{e_n\}_{n=2}^{k-D}$ if necessary, we may assume
that $| \langle e_n, e_2 \rangle | < 1/2$ for $n=3, \cdots, k-2D.$
Continuing this process, we may assume that
$| \langle e_n, e_j \rangle | < 1/2$ for $n=(j+1), \cdots, (k-jD)$
as long as $j+1 \leq k-jD$, that is $j \leq (k-1)/(D+1)$.
If $M = \lfloor (k-1)/(D+1) \rfloor +1$ then the set $\{e_n\}_{n=1}^{M}$ satisfies
$| \langle e_l, e_m  \rangle | <1/2$ whenever 
$l \neq m$ and $l,m \in \{1,2, \cdots, M \}$.
\end{proof}

\begin{lemma} \label{lem:quant}
Let $S \subset \ltwo$ be a relatively compact subset of $\ltwo$.  Fix $C,s,D>0.$
There exists $N \in \N$ such that 
if  $\{e_n\}_{n=1}^{\infty} \subset \ltwo$ with $\|e_n\|_{2}=1$ satisfies one of the
following conditions:
\begin{enumerate}
\item  for every 
$f \in H= \overline{span} \{e_n \}_{n=1}^{\infty}$,
$\dst\left( \sum_{n=1}^{\infty} |\langle f, e_n \rangle |^s \right)^{1/s} \leq C\|f\|_{2},$
\item  for every
$m\in \N$, $\# \{ n \in \N : | \langle e_n, e_m \rangle | \geq 1/2 \} < D,$
\end{enumerate}
and $\{e_n\}_{n=1}^k \subset S$ then $k \leq N$ must hold.  The constant $N$ depends on $S$ and the relevant constants among $C,D,s$.
\end{lemma}

\begin{proof}
Let us first prove that condition (1) implies condition (2).
Indeed, if (1) holds then for all $m\in\N$, 
$\dst\sum_{n=1}^{\infty} | \langle e_n, e_m \rangle |^s \leq C^s.$
Tchebyshev's inequality gives (2) with $D=2^sC^s+1$
$$
\forall m\in \N, \ \ \ \# \{ n\in\N : | \langle e_m, e_n \rangle | \geq1/2 \} \leq 2^s C^s.
$$

Suppose now that (2) holds.  Let $J\subset \{1,2,\cdots,k\}$ be the set obtained when
Lemma \ref{pigeon-lem} is applied to $\{e_n\}_{n=1}^k$.
Since $S$ is relatively compact, the $\ltwo$ closure of $S$, denoted $\overline{S}$, 
is compact.  The open balls $B_{1/2}(x) = \{ y \in \ltwo : \| y-x \|_{2}<1/2\}$ with $x \in S$ provide
an open cover of $\overline{S}$.  So there exist $\{x_n\}_{n=1}^M \subset S$ such that
$\dst\overline{S} \subset \bigcup_{n=1}^M B_{1/2}(x_n)$.
By properties of the set $J$, if $m\neq n$ and $m,n\in J$ then there holds
$$
\|e_n -e_m\|_2^2 = 2[1-\Re(\langle e_n, e_m \rangle)] \geq 2(1-|\langle e_n, e_m \rangle |)> 1.
$$
So each $B_{1/2}(x_n)$ contains at most one
of the $\{e_n\}_{n\in J}$.
However since $$\dst\{e_n\}_{n\in J} \subset \overline{S} \subset \bigcup_{n=1}^M B_{1/2}(x_n),$$ 
we obtain $\#J \leq M$.
Thus, $k \leq M(\lceil D \rceil+1) +1$ and we may take $N=M(\lceil D \rceil+1)+1.$
\end{proof}

In view of Corollary \ref{cor:sh}, Lemma \ref{lem:quant} implies the following theorem.

\begin{theorem}\label{th:210}
Fix $A,C,D,s>0$, $p,q>1$ and $\varphi, \psi \in \ltwo$.
There exist constants $M=M(A,C,D,s,p,q)$ and $N = N(\varphi,\psi,C,D,s)$ such that
if $\{e_n\}_{n=1}^{\infty} \subset L^2(\R)$ with $\|e_n\|_{2}=1$
satisfies either condition (1) or (2) from Lemma \ref{lem:quant}
({\it e.g.} if it is either a Schauder basis or a frame) 
then the sets
\begin{align*}
I & =\big\{n\in \N \,:\ 
|\mu_p(e_n)|\leq A,\ \Delta_p(e_n)\leq A,\  |\mu_q(\widehat{e_n})|\leq A,\mathrm{and}\ \Delta_q(\widehat{e_n})\leq A\big\}, \\
J&=\big\{n\in\N \,:\ |e_n(t)|\leq |\varphi(t)| \ \mathrm{and}\ |\widehat{e_n}(t)|\leq |\psi(t)| \ \hbox{ for a.e. } t \in \R
\big\},
\end{align*}
satisfy $\#I \leq M$ and $\# J \leq N$.
\end{theorem}

For more quantitative bounds on $M,N$
in the cases of orthonormal bases and almost orthonormal Riesz bases, see \cite{JP}.

Note that if $\{e_n\}_{n=1}^{\infty}$ satisfies
condition (1) or (2) of Lemma \ref{lem:quant}
then $\{f_n\}_{n=1}^{\infty}$ defined by $f_n(t) = e^{-2\pi i \mu t}e_n(t)$ with fixed $\mu \in \R$ also satisfies (1) or (2). Also,
$\mu_p(f_n) = \mu_p(e_n)$, $\mu_q(\widehat{f_n}) = \mu_q(\widehat{e_n})-\mu$, 
$\Delta(f_n) = \Delta(e_n)$, $\Delta(\widehat{f_n}) = \Delta(\widehat{e_n})$.
Consequently, for every $\mu\in\R$,
$$
\# \{ n \in \N: 
|\mu_p(e_n)|\leq A,\ \Delta_p(e_n)\leq A,\ |\mu_q(\widehat{e_n})-\mu |\leq A, \ \Delta_q(\widehat{e_n}) \leq A \} \leq M,$$
where $M$ is the constant from Theorem \ref{th:210}, and $M$ does not depend on $\mu$.
This implies the following corollary.

\begin{corollary} \label{cn-cor}
Fix $A>0$ and $p,q>1$.
Suppose that $\{e_n\}_{n=1}^{\infty} \subset \ltwo$ with $\|e_n\|_{2}=1$ satisfies either condition
(1) or (2) of Lemma \ref{lem:quant}, and that, for every $n\in\N$,
$$
|\mu_p(e_n)|\leq A, \ \ \Delta_p(e_n)\leq A, \ \ \Delta_q(\widehat{e_n})\leq A.
$$
There exists a constant $c>0$ such that  $|\mu_q(\widehat{e_n})| \geq cn$ holds for all $n\in\N$.
\end{corollary}

\section{One bounded mean and two bounded variances} \label{mvmain:sec}

The following lemma is a straightforward extension of Lemma 5.5 in \cite{Po}.
\begin{lemma} \label{innerprod-lem}
Fix $p,q>1$.  Suppose $f,g \in \ltwo$ and $\|f \|_{2}=\|g\|_{2}=1$
and that the 
dispersions
$\Delta_p(f), \Delta_q(\widehat{f}), \Delta_p(g), \Delta_q(\widehat{g})$ are all finite.  Then 
$$
|\scal{f,g}|\leq
\frac{ 2^{p/2} ( \Delta_p(f)+\Delta_p(g) ) +2^{q/2} ( \Delta_q(\widehat{f})+\Delta_q(\widehat{g}) )}{|\mu_p(f)-\mu_p(g)|^{p/2}+|\mu_q(\widehat{f})-\mu_q(\widehat{g})|^{q/2}}.
$$
\end{lemma}

We are now in position to prove the following theorem.

\begin{theorem}\label{mv-mainthm}
Fix $A,B,C>0$, $p,q>1$ and $r,s>0$.  
Let $\{e_n\}_{n=1}^{\infty} \subset \ltwo$ with $\|e_n\|_{2}=1$ satisfy 
$$
\forall f \in \ltwo, \ \ \ 
B\left( \sum_{n=1}^{\infty} | \langle f, e_n \rangle |^s \right)^{1/s} \leq 
\|f \|_{2} \leq C\left( \sum_{n=1}^{\infty} | \langle f, e_n \rangle |^r \right)^{1/r}.
$$
If $qr>2$ then it is not possible for $\{e_n\}_{n=1}^{\infty}$ to satisfy
$$
\forall n\in\N, \ \ \ 
|\mu_p(e_n)|\leq A,\ |\Delta_p({e_n})|\leq A, \ \mathrm{and}\ |\Delta_q(\widehat{e_n})|\leq A.
$$
\end{theorem} 

\begin{proof}
By Corollary \ref{cn-cor} there is a constant $c>0$ such that $|\mu(\widehat{e}_n)|\geq c n$
holds for all $n \in \N$.

Let $g(x)=2^{1/4} e^{-\pi x^2}$ and $g_N(x)= g(x-N)$.  Note that $\|g\|_{2}=1$.  
By Lemma \ref{lemma:basicdisp} we have
$\mu_p(g_N) =\mu_p(g)+ N=N$, $\mu_q(\widehat{g_N}) =  \mu_q(\widehat{g})=0$,
 $\Delta_p(g_N) = \Delta_p(g)$, $\Delta_q(\widehat{g_N}) = \Delta_q(\widehat{g})$. 
By Lemma \ref{innerprod-lem}, for $N>A$
$$|\langle g_N, e_n \rangle | \leq 
\frac{2^{p/2} ( \Delta_p(g) + A) + 2^{q/2} (\Delta_q(\widehat{g}) + A)}
{|N- A|^{p/2}+|c n|^{q/2} }.$$
Hence, for an appropriate constant $c_1>0$
$$
1 =  \|g_N \|^{r}_{2} \leq C^r \sum_{n=1}^{\infty} | \langle g_N,e_n \rangle  |^r
\leq c_1 \sum_{n=1}^{\infty} \frac{1}{(|N-A|^{p/2}+|cn|^{q/2})^r}.$$
This gives a contradiction since the assumptions $qr>2$ and $p>0$ imply that
$$\lim_{N\to \infty} \sum_{n=1}^{\infty} \frac{1}{(|N-A|^{p/2}+|cn|^{q/2})^r}=0.$$
\end{proof}

The following corollaries are consequences of (\ref{frame-ineq}) and (\ref{gg-cond}) and 
extend (\ref{mvonb}). 

\begin{corollary} \label{cor-schauder}
Fix $A>0$, $p>1$.  If $\{e_n\}_{n=1}^{\infty}\subset \ltwo$ satisfies $\|e_n\|_{2}=1$ and
$$\forall n\in\N, \ \ \
|\mu_p(e_n)|\leq A,\ |\Delta_p({e_n})|\leq A, \ \ \mathrm{and} \ \ |\Delta_2(\widehat{e_n})|\leq A,$$
then $\{e_n\}_{n=1}^{\infty}$ cannot be a Schauder basis for $\ltwo$.
\end{corollary}

\begin{corollary} \label{cor-frame}
Fix $A>0$, $p,q>1$.  If $\{e_n\}_{n=1}^{\infty}\subset \ltwo$ satisfies $\|e_n\|_{2}=1$ and
$$\forall n \in \N, \ \ \ 
|\mu_p(e_n)|\leq A,\ |\Delta_p({e_n})|\leq A, \ \ \mathrm{and} \ \ |\Delta_q(\widehat{e_n})|\leq A,$$
then $\{e_n\}_{n=1}^{\infty}$ cannot be a frame for $\ltwo$.
\end{corollary}

\section{Exact systems and time-frequency localization} \label{exact:sec}

In this section we shall show that the phenomenon described in Theorem \ref{mv-mainthm} does not
hold for exact systems in $\ltwo$.  We shall make use of the following example.

\begin{examplenum}\label{exam:exact}
Let $g(t)=2^{1/4}e^{-\pi t^2}$ and consider the Gabor system $\mathcal{G}(g,1,1) = \{g_{m,n} \}_{m,n\in \Z}$ defined by
$$\forall m,n, \in \Z, \ \ \  g_{m,n}(t) = e^{2\pi i m t} g(t-n).$$
The unit-norm system $\mathcal{G}_0 = \mathcal{G}(g,1,1) \backslash \{ g_{1,1} \}$ is an
exact system in $\ltwo$, see \cite{F}.
Moreover, $\mathcal{G}_0$ is neither a frame nor a Schauder basis for $\ltwo$.  

Further related examples and uncertainty principles involving exact Gabor systems can be found in \cite{ALS} and \cite{HP,NO}
respectively.
\end{examplenum}

The construction in the following theorem is inspired by \cite{Bo,You2}.
\begin{theorem} \label{exact-thm}  
Fix $p,q>1$ and let $g(t) = 2^{1/4}e^{-\pi t^2}$.  Given $\epsilon>0$, there exists a unit-norm
exact system $\{f_n\}_{n=1}^{\infty}$ in $L^2(\R)$ such that 
\begin{equation} \label{fourbnd}
\forall n \in \N, \ \ \
|\mu_p(f_n)| < \epsilon, \ \ |\mu_q(\widehat{f_n})| < \epsilon,
\ \ \Delta_p(f_n) < \Delta_p(g)+\epsilon, \ \ \Delta_q(\widehat{f_n}) < \Delta_q(\widehat{g}) + \epsilon,
\end{equation}
and such that
$\varphi(t) = \sup_{n\in \N}  |f_n(t)|$ and
$\psi(\xi) = \sup_{n\in \N}|\widehat{f_n}(\xi)|$ satisfy $\varphi, \psi \in L^2(\R).$
\end{theorem}

\begin{proof} Note that the system $\mathcal{G}(g,1,1)$ from Example \ref{exam:exact} satisfies
$\|g\|_{2}=1$ and that $\mu_p(g_{m,n})=\mu_p(g)+n=n$, $\mu_q(\widehat{g_{m,n}}) =m$, $\Delta_p(g_{m,n}) =\Delta_p(g),$ and $ \Delta_q(\widehat{g_{m,n}}) = \Delta_q(g)$.
We may enumerate the system $\mathcal{G}_0$ from Example \ref{exam:exact} as
$\mathcal{G}_0 = \{ e_n \}_{n=1}^{\infty}$ so that $e_1=g_{0,0}=g.$ 

Let $\{\alpha_n\}_{n=1}^{\infty}\subset \R$ satisfying $0<\alpha_n<1$ be a sequence to be specified below and define
$$
\forall n \in \N,  \ \ \ f_n=\frac{e_1+\alpha_n e_{n+1}}{\|e_1+\alpha_n e_{n+1}\|_{2}}.
$$
Note that
$$
\widehat{f_n} = 
\frac{\widehat{e_1} + \alpha_n \widehat{e_{n+1}}}{\|\widehat{e_1}+\alpha_n \widehat{e_{n+1}}\|_{2}}.
$$
Using Proposition \ref{prop:dispconv}, for each $n\in \N$ we select $\alpha_n$ sufficiently
small so that $0<\alpha_n<2^{-n}$ and such that
$$
| \mu_p(f_n)|=
| \mu_p(f_n) - \mu_p(e_1) | < \epsilon \ \ \ 
\hbox{ and } \ \ \ | \mu_q(\widehat{f_n})|=| \mu_q(\widehat{f_n}) - \mu_q(\widehat{e_1}) | < \epsilon,$$
and
$$\Delta_p(f_n) < \Delta_p(e_1) + \epsilon \ \ \ \hbox{ and } \ \ \ 
 \Delta_q(\widehat{f_n}) < \Delta_q(\widehat{e_1}) + \epsilon.$$
Further, since $1/2\leq1-2^{-n} <1-\alpha_n \leq \|e_1 + \alpha_n e_{n+1}\|_2$ we have
$|f_n| \leq 2(|e_1| + 2^{-n} |e_{n+1}|)$ and hence
$|\varphi| \leq 2|e_1|+2\sum_{n=1}^{\infty} 2^{-n} |e_{n+1}|.$
Thus $\|e_n\|_2=1$ gives
$$\|\varphi \|_2 \leq 2\|e_1\|_2+2\sum_{n=1}^{\infty} 2^{-n} \|e_{n+1}\|_2 \leq 4.$$
Hence $\varphi \in \ltwo$.  A similar computation shows that $\psi \in \ltwo$.

Finally, to show that $\{f_n\}_{n=1}^{\infty}$ is exact we must show that it is complete and minimal.  
First, note that for each $n \in \N$, $e_n \in \overline{span} \{f_n\}_{n=1}^{\infty}$.
Indeed, $\|e_1+\alpha_ne_{n+1}\|_2 f_n=e_1+\alpha_n e_{n+1}\to e_1$ as $n \to \infty$, and also for each $n \in \N$, $e_{n+1}=\frac{1}{\alpha_n}(\|e_1+\alpha_ne_{n+1}\|_2 f_n-e_1)$.
Thus, $\ltwo = \overline{span} \{e_n\}_{n=1}^{\infty} \subset
\overline{span} \{f_n\}_{n=1}^{\infty},$ and $\{f_n\}_{n=1}^{\infty}$ is complete in $\ltwo$.

To see that $\{f_n\}_{n=1}^{\infty}$ is minimal suppose that
$f_N \in \overline{span} \{f_n: n \neq N \}.$
Since, as above, $e_1 \in \overline{span } \{f_n : n > N \}$, it would follow that
$e_{N+1} \in \overline{span} \{f_n : n \neq N\}$.  However this contradicts the
minimality of $\{e_n\}_{n=1}^{\infty}$ since
$\overline{span} \{f_n : n \neq N\} \subset \overline{\rm span} \{e_n : n \neq (N+1)\}$
and $e_{N+1} \notin \overline{span} \{e_n : n\neq (N+1)\}$.
\end{proof}

By Corollary \ref{cor-schauder} the exact system $\{f_n\}_{n=1}^{\infty}$ constructed in 
Theorem \ref{exact-thm} is not a Schauder basis.   This can also be seen using a direct computation
to show that $\{f_n\}_{n=1}^{\infty}$ does not have a basis constant.

\subsection*{Ackowledgments}  Ph. Jaming and A. Powell were both partially supported by the
ANR project AHPI {\it Analyse Harmonique et Probl\`emes Inverses}.

A. Powell was supported in part by NSF Grant DMS-0811086.
Portions of this work were completed during 
visits to the Universit\'e d'Orl\'eans (Orl\'eans, France),
the Academia Sinica Institute of Mathematics (Taipei, Taiwan),
and the City University of Hong Kong (Hong Kong, China).  This author is grateful
to these institutions for their hospitality and support.

The authors thank S. Nitzan and H. Shapiro for helpful discussions related to the material.

\bibliographystyle{amsplain}

\end{document}